\documentclass[10pt]{article}
\usepackage{charter}
\usepackage[margin=3cm]{geometry}
\usepackage{amsmath,amssymb}
\usepackage[mathscr]{euscript}
\usepackage[usenames,dvipsnames]{color}
\usepackage{pstricks}
\usepackage{graphicx}
\usepackage[all]{xy}
\usepackage{tabularx}
\usepackage{bbm}
\usepackage{ntheorem}
\newtheorem{thm}{Theorem}[section]

\newtheorem{defi}[thm]{Definition}
\newtheorem{lemm}[thm]{Lemma}
\newtheorem{coro}[thm]{Corollary}

\newcommand{\bbox}{\normalsize {}%
        \nolinebreak \hfill $\blacksquare$ \medbreak \par}
\newenvironment{proof}[1][\unskip]{\noindent\emph{Proof #1:} }{\bbox\vspace{0,15cm}}

\newcommand{\D}{\mathrm D}
\newcommand{\wh}[1]{\widehat{#1}}

\newcommand{\ii}{\mathrm{i}}

\newcommand{\ol}[1]{\overline{#1}}
\newcommand{\lto}{\longrightarrow}
\newcommand{\EL}{\mathrm{E}_{\leqslant\Lambda}}

\usepackage{hyperref}
\hypersetup{colorlinks=true,    linkcolor=blue,    filecolor=magenta,      citecolor=blue,    urlcolor=cyan}

\makeatletter
\newcommand*{\defeq}{\mathrel{\rlap{%
                     \raisebox{0.3ex}{$\m@th\cdot$}}%
                     \raisebox{-0.3ex}{$\m@th\cdot$}}%
                     =}
\makeatother

\title{Spectrally cut-off GFF, regularized $\Phi^4$ measure, and reflection positivity}
\author{I. Bailleul\footnote{Univ Brest, CNRS UMR 6205, Laboratoire de Math\'ematiques de Bretagne Atlantique, France. E-mail: ismael.bailleul@univ- brest.fr. Partial support from the ANR-11-LABX-0020-01 Labex CHL and the ANR project Smooth ANR-22-CE40-0017 is acknowledged.}, N.V. Dang\footnote{Sorbonne Universit\'e – Universit\'e de Paris, CNRS, UMR 7586, Paris, France. Institut Universitaire de France.
E-mails: nguyen-viet.dang@imj-prg.fr, jiasheng.lin@imj-prg.fr, gaetan.leclerc@imj-prg.fr}, L. Ferdinand\footnote{Laboratoire de Physique des 2 infinis Ir\`ene Joliot-Curie, UMR 9012, Universit\'e Paris-Saclay, Orsay, France. E-mail: lferdinand@ijclab.in2p3.fr}, G. Leclerc$^{\textrm{\textdagger}}$ and J. Lin$^{\textrm{\textdagger}}$}

\begin{document}

\maketitle

\begin{abstract}
We argue that the spectrally cut-off Gaussian free field $\Phi_\Lambda$ on a compact Riemannian manifold or on $\mathbb{R}^n$ cannot satisfy the spatial Markov property. Moreover, when the manifold is reflection positive, we show that $\Phi_\Lambda$ fails to be reflection positive. We explain the difficulties one encounters when trying to deduce the reflection positivity property of the measure exp$(-\|\rho\Phi_\Lambda\|_{L^4}^4) \mu_{\text{GFF}}(d\Phi)$ from the reflection positivity property of the Gaussian free field measure $\mu_{\text{GFF}}$ in a naive way. These issues are probably well-known to experts of constructive quantum field theory but to our knowledge, no detailed account can be found in the litterature. Our pedagogical note aims to fill this small gap. 
\end{abstract}

\section{Introduction and context}

\noindent \textbf{1.1 Markov property and reflection positivity of a random field.} Let $(M,g)$ be a smooth, closed, compact Riemannian manifold or $M=\mathbb{R}\times \Sigma$ where $\Sigma$ is a complete Riemannian manifold endowed with the product metric $dt^2+g_\Sigma$. We  denote by $\Delta_g$ the positive Laplace-Beltrami operator acting on $C^\infty(M)$. In case $M$ is compact, we write $(e_\lambda)_{\lambda\in \sigma(\Delta_g)}$ for the $L^2$ basis of eigenfunctions of $\Delta_g$ and 
$$
\mathrm E_{\leqslant \Lambda} \defeq \mathrm{span}(e_\lambda)_{\lambda\leqslant\Lambda^2}\subset C^\infty(M).
$$
The massive Gaussian Free Field (GFF in the sequel) $\Phi$ of law $\mu$ is defined
on $M$ compact as
the random series
$$ \Phi\defeq\sum_{\lambda\in\sigma(\Delta_g)}\frac{c_\lambda}{(\lambda+1)^{\frac{1}{2}}} \, e_\lambda  $$
where the coefficients $c_\lambda$ are independent, identically distributed, random variables with Gaussian distribution $\mathcal{N}(0,1)$. On a cylinder $M=\mathbb{R}\times \Sigma$ the massive Gaussian free field is defined as the unique Gaussian process $\Phi$ indexed by $H^{-1}(M)$ with covariance
$$\mathbb{E}\big[\Phi(f)\Phi(h)\big] = \left\langle f,h\right\rangle_{H^{-1}(M)}.$$
Pick $\varepsilon>0$. One can realize the Gaussian process $\Phi$ as a random variable with values in $H^{\frac{2-d}{2}-\varepsilon}(M)$. Denoting by $\mu_{\text{GFF}}$ the law of $\Phi$, we will work in the sequel with the canonical probability space $\big(H^{\frac{2-d}{2}-\varepsilon}(M), \mathcal{O}, \mu_{\text{GFF}}\big)$, where $\mathcal{O}$ stands for the Borel $\sigma$-algebra of $H^{(2-d)/2-\varepsilon}(M)=\Omega$. One can take $\Phi(\omega)=\omega$ for every $\omega\in\Omega$.

\begin{defi}
In both cases of $M$ compact or a cylinder it is well-known that $\Delta_g:C^\infty_c(M)\subset L^2(M)\rightarrow L^2(M) $ is essentially self-adjoint hence it admits a well-defined functional calculus. 
For $\Lambda>0$ we use the notation 
$$
\Pi_\Lambda\defeq\mathbbm{1}_{[0,\Lambda^2]}\big(\Delta_g\big)
$$ 
for the sharp spectral projector -- so when $M$ is compact one has $\Pi_\Lambda(\Delta_g):\mathcal D'(M)\rightarrow \EL$. Note that $\Pi_\Lambda$ is self-adjoint. We define the spectrally cut-off GFF 
as
$$\Phi_\Lambda\defeq\Pi_\Lambda(\Phi)\,, $$ 
which is a random smooth function. 
\end{defi}

Note that we could take a smooth compactly supported cut-off $\Psi\big(\frac{  \Delta_g}{\Lambda^2}\big)$ in the sequel, without loss of generality except in Section \ref{sec2}. We will use the notion of smooth functionals  on a locally convex space such as $\mathcal{D}^\prime(M),C^\infty(M)$ or on some Sobolev space $H^s(M)$ as discussed in~\cite{DBLR}, along with the notion of support of a smooth functional. Both notions are recalled hereafter, we start with smooth functionals~:
\begin{defi}
Let $E$ be a locally convex space and $U\subset E$ an open subset. A function $F:U\mapsto \mathbb{R}$ is \textbf{smooth} if for every $k\in \mathbb{N}$, $x\in U$, every $(h_1,\dots,h_k)\in E^k$, the limit
$$  \frac{\partial^kF\big(x+t_1h_1+\dots+t_kh_k\big)}{\partial t_1 \dots\partial t_k}|_{(t_1,\dots,t_k)=0}=\D^k_xF\big(h_1,\dots,h_k\big) $$ exists and $\D^kF:U\times E^k\mapsto \mathbb{R}$ is jointly continuous as a function of its arguments and linear in $h_1,\dots,h_k$.   
\end{defi}
We next discuss the notion of support:
\begin{defi}\label{def:suppfonc}
In what folllows $E=C^\infty(M),\mathcal{D}^\prime(M)$ or some Sobolev space $H^s(M)$ of distributions on some smooth manifold $M$.
Let $U\subset E$ be an open subset. 
The \textbf{support of some smooth functional} $F: U\rightarrow \mathbb R$, denoted by $\mathrm{supp}(F)$, is the set containing those points $x\in M$ such that for every neighbourhood $A_x\subset M$ of $x$, there exists $\phi,\psi\in U$ that verify $F(\phi)\neq F(\psi)$ and $\mathrm{supp}(\phi-\psi)\subset A_x$.
\end{defi} 

We next introduce the notion of reflection positive manifold  following the work of Jaffe \& Ritter~\cite{JR07,JR08,JRmon}:
\begin{defi}  
We say that the manifold $M$ is \textbf{reflexion positive} if there exists an \textbf{isometric involution} $\Theta:M\mapsto M$ which admits as invariant subset a submanifold $\Sigma$ of dimension $d-1$ such that $M\setminus \Sigma$ is the disjoint union of two open manifolds $M_+$ and $M_-$ that both verify $\partial \overline{M_+}=\partial \overline{M_-}=\Sigma$. One can therefore think of $M$ as the disjoint union
\begin{align*}
    M=M_+\cup\Sigma\cup M_-\,,\,
    \Theta|_\Sigma=\mathrm{Id}_\Sigma\,,\, \Theta (M_\pm) = M_\mp.
\end{align*}
\end{defi}

\begin{figure}[ht]
\includegraphics[scale=0.55]{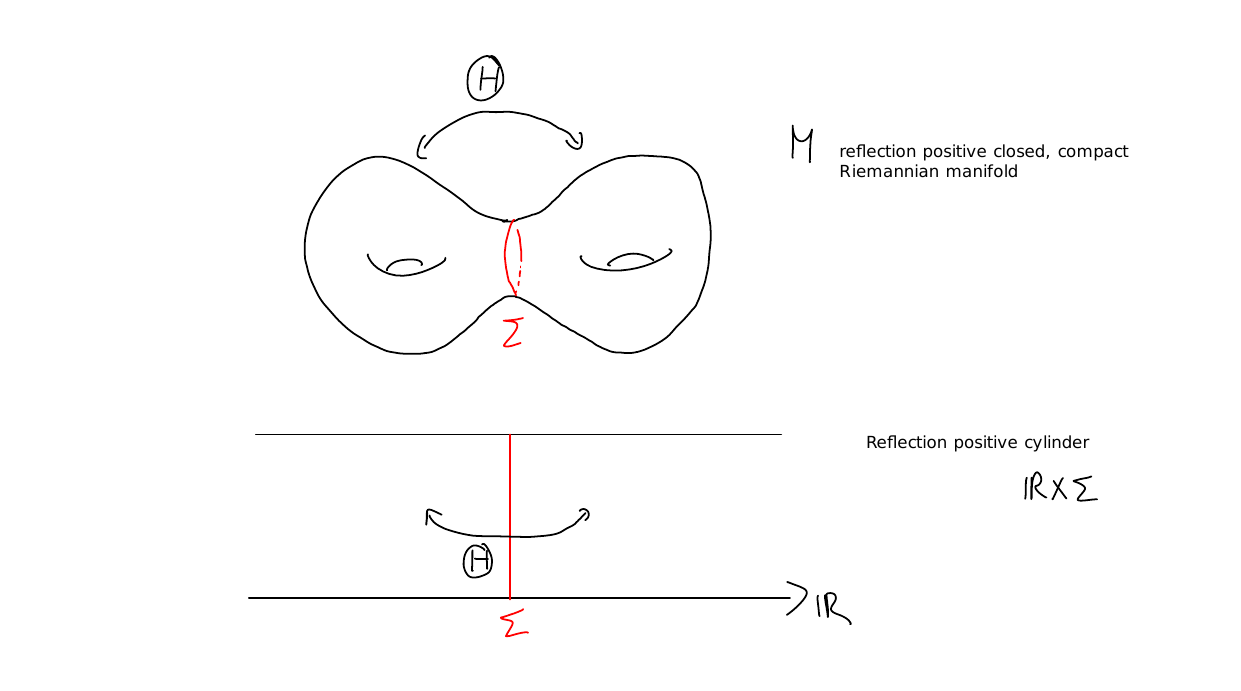}
\centering
\caption{\label{f:RP_space}.}
\end{figure}

\noindent For clarity we write below $\overline{M}_\pm$ for $\overline{M_\pm}$ and refer to figure~\ref{f:RP_space}. The spaces of constant curvature $\mathbb R^d$, $\mathbb S^d$, and $\mathbb H^d$ are reflection positive. In the sequel, we will mostly be interested in the cases $M=\mathbb R\times\Sigma,\,\mathbb S^d$, for which we precise the setting. The cylinder $\mathbb R\times\Sigma$ is reflection positive with respect to $\{0\}\times\Sigma$. The sphere $\mathbb S^d$ is reflection positive with respect to any of its equators. Whenever a manifold $M$ is reflection positive with isometric involution $\Theta$, the map $\Theta$ acts via pullback on $C^\infty(M)$ and $\mathcal{D}'(M)$ and hence on functionals on $M$ and $\mathcal{D}'(M)$-valued random variables. Note that for a distribution $\omega$
\begin{equation} \label{EqSymmetrySpectralProjectors}
\Pi_\Lambda(\omega)_{\vert\overline{M}_-} = \big(\Pi_\Lambda(\omega) \circ\Theta\big)_{\vert\overline{M}_+} = \Pi_\Lambda(\omega\circ\Theta)_{\vert\overline{M}_+}.
\end{equation}


Let $A$ be a closed subset of $M$. We denote by $W^{-1}_A(M) $ the subset of elements in the Sobolev space $W^{-1}(M)$ which are supported in $A$. Let $\Psi$ be a random distribution defined on our probability space $\Omega$, with law $\mu_\Psi$ and associated expectation operator $\mathbb{E}_\Psi[\cdot]$. We introduce a sub $\sigma$-algebra of $\mathcal{O}$
\begin{align*}    
\sigma(\Psi;A) \defeq \sigma\Big(  \Psi(f) ; f\in W^{-1}_A(M)\Big).
\end{align*}
We implicitly assume here that we only consider those fonctions $f\in  W^{-1}_A(M)$ for which $\Psi(f)$ is well-defined. One can think of $\sigma(\Psi;A)$ as the $\sigma$-algebra generated by the observables that are only sensitive to the fluctuation of the field $\Psi$ on the set $A$.

\begin{defi}
A random process $\Psi$ as above has the \textbf{Markov property} if for \textit{any} closed subsets $A, B$ of $M$ such that $A\cap B$ has an empty interior, see figure~\ref{f:Markov},  and for any random variable $F\in L^2$ that is $\sigma(\Psi;B)$-measurable, one has
\begin{equation*}
  \mathbb{E}\big[F | \sigma(\Psi;A)\big] = \mathbb{E}\big[F | \sigma(\Psi;\partial A)\big].
  \label{}
\end{equation*}
\end{defi}

\begin{figure}[ht]
\includegraphics[scale=0.4]{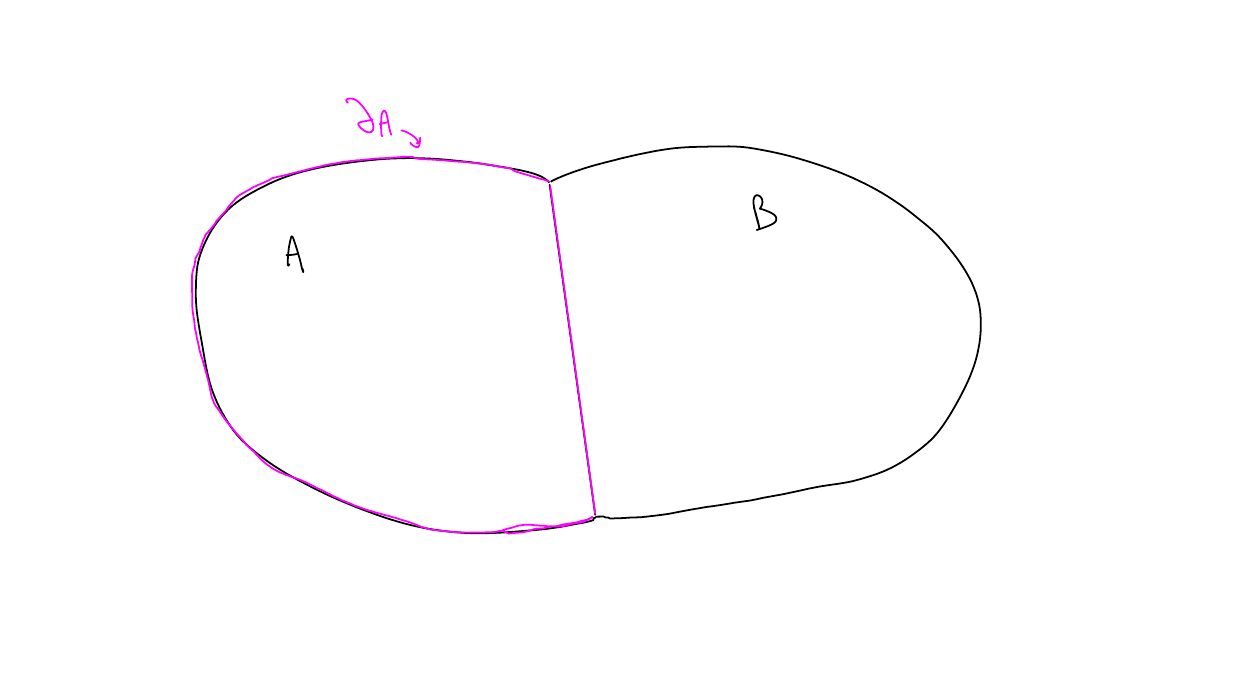}
\centering
\caption{\label{f:Markov}.}
\end{figure}

The Markov property was introduced by Nelson in \cite{Nel1},\cite[p.~224-225]{Nel2}; he proved that $\mu_{\text{GFF}}$ verifies this property ~\cite[p.~225]{Nel2}.

\begin{defi}
A random process $\Psi$ on a reflection positive manifold $M$ is \textbf{reflection positive} if for any $\mathbb{C}$-valued function $F$ in $ L^2\big(\mathcal{D}^\prime(M),\sigma(\Psi;\overline{M}_+),\mu_\Psi\big)$ it holds
\begin{align*}
    \mathbb E_\Psi\big[(\overline{\Theta F}) F\big] \geqslant 0.
\end{align*}
\end{defi}

\noindent Reflection positivity is one of the axioms introduced by Osterwalder and Schrader \cite{OS}; it is a crucial condition to recover a Lorentzian Quantum Field Theory from a Euclidean Quantum Field Theory. It is connected to the Markov property by the following fact

\begin{lemm} [\cite{Dim}, Theorem 2] \label{LemMarkovImpliesRP}
On a reflection positive Riemannian manifold $(M,g)$ a random field that has the Markov property is reflection positive.
\end{lemm}

\smallskip

\noindent \textbf{1.2 The regularized $\Phi^4_3$ measure built from the cut-off interaction.} The $\Phi^4_3$ measure $\nu$ is a probability measure on $\mathcal D'(M)$ with $(M,g,\Theta)$ a closed three dimensional reflection positive manifold that heuristically reads as
\begin{align} \label{EqPhi43Density}
\nu(d\Phi)\propto e^{- c\Vert\Phi\Vert^4_{L^4(M)}} \, \mu_{\text{GFF}}(d\Phi).
\end{align}
 We use the non-conventional notation $c$ for the coupling constant. Since the dimension of $M$ is greater then or equal to 2 the GFF is not supported on $L^p$ functions and the $L^4$ norm of $\Phi$ is almost surely infinite. The formal expression \eqref{EqPhi43Density} is thus meaningless. The probability measure $\nu$ can therefore only be defined as the weak limit of a sequence of approximations $( \nu_{\rho,\Lambda})_{\Lambda\geqslant0}$. A choice which is often made is to consider
\begin{align}\label{eqnu}     
 \nu_{\rho,\Lambda}(d\Phi)\propto e^{- c\Vert \rho\Phi_\Lambda\Vert^4_{L^4(M)} - c a_\Lambda\Vert \rho\Phi_\Lambda\Vert^2_{L^2(M)}} \, \mu_{\text{GFF}}(d\Phi)\,,
\end{align}
where $\rho\in C^\infty_c(M)$ is identically equal to $1$ on a large compact set of $M$ and $(a_\Lambda)_{\Lambda\geqslant0}$ is a suitably chosen sequence of real number that is divergent at large $\Lambda$. 

On $M$ the $\Phi^4_3$ measure $\nu$ is a good starting point to construct a QFT provided $M$ is reflection positive and $\nu$ is proved to be reflection positive too. When $\nu$ is constructed as a limit of a sequence of approximations the only way to prove reflection positivity of $\nu$ is to show that the measures $\nu_{\rho,\Lambda}$ are reflection positive for all $\Lambda$, in which case $\nu$ inherits the reflection positivity property of its approximations.

In \cite{AK} Albeverio \& Kusuoka  use the approximation \eqref{eqnu} to construct the $\Phi^4_3$ measure on $\mathbb R^d$ and claim that the measures $( \nu_{\rho,\Lambda})_{\Lambda\geqslant0}$ are reflection positive -- from which they conclude that the limiting measure is also reflection positive. They justify their claim in~\cite[section 7.2 p.~82-83]{AK} as follows. Denoting for any closed subset $A\subset M$
\begin{align} \label{EqGWholeSpace}
G_{A,\rho,\Lambda} \defeq e^{- c\Vert\rho\Phi_\Lambda\Vert^4_{L^4(A)} -  c a_\Lambda\Vert\rho\Phi_\Lambda\Vert^2_{L^2(A)}} \leq 1,
\end{align} 
one has  $G_{M,\rho,\Lambda} = G_{\overline{M}_-,\rho,\Lambda}G_{\overline{M}_+,\rho, \Lambda}$. They first claim that $G_{A,\rho,\Lambda} \in L^2\big(H^{\frac{2-d}{2}-\varepsilon}(M),\sigma(\Phi;A),\mu_{\text{GFF}}\big)$ and for 
$F\in L^2\big(H^{\frac{2-d}{2}-\varepsilon}(M),\sigma(\Phi;\overline{M}_+),\mu_{\text{GFF}}\big)$ they reexpress in~\cite[section 7.2 last line of p.~82 and first two lines of p.~83]{AK}
the expectation
\begin{align*}
\mathbb E_{{\color{black} \nu_{\rho,\Lambda}}}\big[(\overline{\Theta F}) F\big] =\frac{1}{Z_\Lambda} \mathbb E_{\mu_{\text{GFF}}}\big[G_{\overline{M}_-,{\color{black} \rho},\Lambda} \, G_{\overline{M}_+,{\color{black} \rho},\Lambda}\,(\overline{\Theta F}) F\big]
\end{align*}
with $Z_{{\color{black} \rho},\Lambda}\defeq\mathbb E_{\mu_{\text{GFF}}}[ G_{M,{\color{black} \rho},\Lambda}]$, under the form 
$$
\mathbb E_{\mu_{\text{GFF}}}\big[G_{\overline{M}_-,{\color{black} \rho},\Lambda} \, G_{\overline{M}_+,{\color{black} \rho},\Lambda}\,(\overline{\Theta F}) F\big] \overset{\eqref{EqSymmetrySpectralProjectors}}{=} \int F(\omega) \, G_{\overline{M}_+,{\color{black} \rho},\Lambda}(\omega) \,\overline{F(\omega\circ\Theta)}  \, G_{\overline{M}_+,{\color{black} \rho},\Lambda}(\omega\circ\Theta) \, \mu_{\text{GFF}}(d\omega).
$$
The reflection positivity of ${\color{black} \nu_{\rho,\Lambda}}$ would then stem from the reflection positivity of the measure $\mu_{\text{GFF}}$ applied to $FG_{\overline{M}_+,\Lambda}$, provided $G_{\overline{M}_+,{\color{black} \rho},\Lambda}$ and $FG_{\overline{M}_+,{\color{black} \rho},\Lambda}$ are both in $L^2(H^{\frac{2-d}{2}-\varepsilon}(M),\sigma(\Phi;\overline{M}_+),\mu_{\text{GFF}})$. However we make the following observation.  \textit{For any closed subset $A\subsetneq M$ of non-empty interior and $\Lambda$ sufficiently large the random variable $G_{A,{\color{black} \rho},\Lambda}$ is not $\sigma(\Phi;A)$-measurable.} In fact this will follow from a more general claim which is proved in Section~\ref{sec1}~:

\begin{thm} \label{thm1}
Let $M$ be a smooth complete Riemannian manifold of dimension $d$. Let $A\subsetneq M$ be a closed subset with nonempty interior, and $G$ a smooth functional on $C^\infty(M)$ such that the random variable $G\circ \Pi_\Lambda(\Phi)$ is $\sigma(\Phi;A)$-measurable. Then the function $G\circ \Pi_\Lambda$ on $H^{\frac{2-d}{2}-\varepsilon}(M)$ is constant.
\end{thm}

In the particular case we are interested in Theorem \ref{thm1} tells us that if $G_{A,{\color{black} \rho},\Lambda}(\Phi)$ were indeed $\sigma(\Phi;A)$-measure then the function $G_{A,{\color{black} \rho},\Lambda}(\cdot)$ on $H^{\frac{2-d}{2}-\varepsilon}(M)$ would be constant. We obtain a contradiction as the polynomial functional $f\in H^{\frac{2-d}{2}-\varepsilon}(M)\mapsto {\color{black} c}\Vert{\color{black} \rho}f_\lambda \Vert^4_{L^4(A)} + {\color{black} ca_\Lambda}\Vert{\color{black} \rho}f_\lambda \Vert^2_{L^2(A)}$ should then constant, giving by rescaling a non-constant polynomial function of degree four of a real variable. Now since $G_{\overline{M}_\pm,{\color{black} \rho},\Lambda}$ is not $\sigma(\Phi;\overline{M}_\pm)$-measurable one \textbf{cannot apply the spatial Markov property} to write the equality
$$ 
\mathbb E_{\mu_{\text{GFF}}}\big[G_{\overline{M}_-,{\color{black} \rho},\Lambda}G_{\overline{M}_+,{\color{black} \rho},\Lambda} (\overline{\Theta F}) F\big] = \mathbb E_{\mu_{\text{GFF}}}\Big[ \mathbb{E}_{\mu_{\text{GFF}}}\left[ G_{\overline{M}_-,{\color{black} \rho},\Lambda}\overline{\Theta F} \big| \sigma(\Phi;\Sigma) \right] \mathbb{E}_{\mu_{\text{GFF}}} \left[ G_{\overline{M}_+,{\color{black} \rho},\Lambda} F \big| \sigma(\Phi;\Sigma) \right]\Big].
$$
This seems to indicate that the discussion in \cite[Section 7.2 last line of p.82 and first two lines of p.83]{AK} might need some more justifications. {\color{black} This seems necessary as we prove in Theorem \ref{ThmNotRFRegularizedPhi43} in Section \ref{SectionCounterexample} that the regularized $\Phi^4_3$ measures $\nu_{\rho,\Lambda}$ are \textit{\textbf{not reflection positive}} for a small enough coupling constant $c>0$ in \eqref{eqnu}.} Moreover we give hereafter three arguments that show that the cut-off free field $\Phi_\Lambda$ can by no means be neither Markov nor reflection positive, neither on cylinders like $\mathbb R^d$ nor on compact manifolds.  (It seems that the authors of \cite{AK} erroneously claim that $\Phi_\Lambda$ has the Markov property when write that ``Similarly to the proof of Theorem 5 in [156], we can prove that $P_N\Phi$ under the free field measure $\mu_0$ is a Markov field'', at the beginning of their Section 7.2.) This also confirms that any approach involving spectral cut-offs is unlikely to be efficient in order to prove the reflection positive of the $\Phi^4_3$ measure.
 
\medskip
 
\noindent \textbf{1.3 Organization of the work.} We prove Theorem~\ref{thm1} in Section~\ref{sec1}. This somehow shows that the spatial Markov property for the cut-off GFF is not well-posed since we prove that the only $\sigma(\Phi;\overline{M}_+)$-measurable random variables of the form $G(\Phi_\Lambda)$ where $G$ is smooth are constants. In Section \ref{sec2} we give a counter-example showing that the cut-off free field the cylinder $\mathbb R\times \Sigma$ is not reflection positive, by constructing a function on which the bilinear form associated to the cut-off covariance acts negatively, see Theorem~\ref{thm2} below. In Section \ref{sec3} we show that the cut-off free field on compact manifolds also has a covariance which is not reflection positive, by constructing a counter-example, see Theorem~\ref{thm41} below. Given Lemma \ref{LemMarkovImpliesRP} the results of these sections show that the cut-off regularized Gaussian free field does not have the Markov property in these settings. All our arguments point toward some possible issue in the proof of reflection positivity of the $\Phi^4_3$ measure proposed in \cite{AK}. Note that in the flat case $M=\mathbb R^3$, it is a well-established fact that the $\Phi^4_3$ measure is reflection positive, since it can be constructed as the limit of a sequence of Gibbs measure on the lattice \cite{GH}, where regularized measures are reflection positive. However this construction does not generalize to the case of a compact Riemannian manifold $M$ where lattice regularization is not an option. Hence the proof of the reflection positivity of the $\Phi^4_3$ measure on the sphere $\mathbb S^3$ is still an open problem. {\color{black} We prove  in Section \ref{SectionCounterexample} that the regularized $\Phi^4_3$ measures $\nu_{\rho,\Lambda}$ are not reflection positive for a small enough coupling constant $c>0$.}

\medskip

\noindent \textbf{1.4 Acknowledgements}
We would like to thank D. Benedetti, M. Gubinelli, A. Mouzard, T.D.T\^o, M. Wrochna for their interest and comments on the present note. N.V.D acknowledges the support of the Institut Universitaire de France.
The authors would like to
thank the ANR grant SMOOTH ”ANR-22-CE40-0017” and QFG ”ANR-20-CE40-0018” for support.

\section{The cut-off GFF admits only trivial smooth localized observables}\label{sec1}

We prove Theorem \ref{thm1} in this section. Our proof is based on the result due to Lebeau-Robbiano~\cite{LR95}, Jerison-Lebeau~\cite{JL99}, Lebeau-Zuazua~\cite{LZ98} which asserts that the zero set (also called \textbf{nodal set}) of linear combinations of Laplace eigenfunctions always has empty interior:
\begin{lemm}[Nodal sets of linear combinations of eigenfunctions]\label{thmnodal}
Let $(M,g)$ be a smooth closed compact Riemannian manifold.
Then for any $\Lambda\in(0,\infty)$ and any nontrivial finite 
linear combination of the eigenfunctions of $\Delta_g$ that we denote by $f\in\EL$, the zero set $Z_f=\{x\in M|f(x)=0\}$ of the function $f$ has \textbf{empty interior}. 
\end{lemm}

\begin{proof}
Assume there exists $f\in \EL$ such that $\Vert f\Vert_{L^2(M)}=1$ and $Z_f$ contains an open subset $U$.
Then the Lebeau-Robbiano spectral inequality~\cite[Thm 14.6 p.~230]{JL99}
states that given $U\subset M$, there exists constants $C,K>0$ such that for all $\varphi\in \EL$, we have an inequality of the form:
$$ \Vert \varphi\Vert_{L^2(M)} \leqslant Ce^{K\Lambda}\Vert \varphi\Vert_{L^2(U)}\,, $$
where $C,K$ do not depend on $\varphi$.
Setting $\varphi=f$ yields
$$\Vert f\Vert_{L^2(M)} \leqslant Ce^{K\Lambda}\Vert f\Vert_{L^2(U)}=0\,, $$
since we assumed $f|_U=0$. This yields a contradiction with the assumption $\Vert f\Vert_{L^2(M)}=1$.
\end{proof}

\begin{coro} \label{Lemm21}
We assume $(M,g)$ is a smooth closed compact Riemannian manifold or $M=\mathbb{R}^n$ with the flat metric and $B$ is some open subset $B\subsetneq M $.
Let $T\in \mathcal{D}^\prime(M)$ be a distribution such that for all $f\in C_c^\infty(B)$, $T(\Pi_\Lambda f)=0$. Then $\Pi_\Lambda T=0$.
\end{coro}

\begin{proof}
    Since $\Pi_\Lambda$ is self-adjoint, we have
\begin{align*}
    T(\Pi_\Lambda f)=\langle \Pi_\Lambda T,f\rangle_{L^2(M)}.
\end{align*}
Therefore, $\Pi_\Lambda T\in E_{\leqslant \Lambda}$ which is vanishing on the interior of $B$. Hence, it is null by Lemma~\ref{thmnodal}.

In the case $M=\mathbb{R}^n$, we repeat the exact same argument and conclude using the fact that $\Pi_\Lambda T$ is analytic by the Paley-Wiener Theorem hence null if it vanishes on the open set $B$.
\end{proof}

We can now give the proof of Theorem \ref{thm1}. We work by contradiction. Set $F=G\circ \Pi_\Lambda$ which is a smooth functional on $H^{\frac{2-d}{2}-\varepsilon}(M)$. Assuming that $F$ is $\sigma(\Phi;A)$-measurable, for every direction $h\in C^\infty_c(A^c)$ we have the identity
$$ 
F(\Phi+h)=F(\Phi) 
$$
for $\Phi$ in a set $\Omega_h\subset\Omega$ of probability $1$ that depends on $h$. Now we would like to go from an $h$-dependent almost sure statement to a deterministic statement meaning $F(\Phi+h)-F(\Phi)=0$ for all distribution $\Phi \in H^{\frac{2-d}{2}-\varepsilon}(M)$. Assume by contradiction that there is some $\varphi\in H^{\frac{2-d}{2}-\varepsilon}(M)$ such that $F(\varphi+h)-F(\varphi)\neq 0$. Assume without loss of generality that $F(\varphi+h)-F(\varphi)=L >0$. The functional $$\Phi\in H^{\frac{2-d}{2}-\varepsilon}(M)\mapsto F(\Phi+h)-F(\Phi)$$ is smooth, hence by continuity there is some open subset $\varphi \ni U_\varphi$ of $H^{\frac{2-d}{2}-\varepsilon}(M)$ such that $0<\frac{L}{2}\leqslant  F|_{U_\varphi}\leqslant \frac{3L}{2}$. Then there is some $R>0$ for which we have the inclusion of the closed ball $B(\varphi,R)\subset U_\varphi$ for the $H^{\frac{2-d}{2}-\varepsilon}(M)$ topology. Denote by $\chi_{B(\varphi,R)}$ the indicator function of $B(\varphi,R)$. One then has
$$
0=\mathbb{E}_{\mu_{\mathrm{GFF}}}\left[ \left(F(\cdot+h)-F\right)\chi_{B(\varphi,R)} \right] \geqslant \frac{L}{2}\mu\left( B(\varphi,R)\right) 
$$
the first equality follows from the almost sure vanishing of $F(\cdot+h)-F$, which implies that 
$$
\mu\big( B(\varphi,R)\big) = 0.
$$ 
But it is an important fact about the massive GFF that $\mu_{\text{GFF}}$ has \textbf{full support} in $H^{\frac{2-d}{2}-\varepsilon}(M)$ since the Cameron-Martin space $H^1(M)$ is everywhere dense in $H^{\frac{2-d}{2}-\varepsilon}(M)$ for the $H^{\frac{2-d}{2}-\varepsilon}(M)$ topology \cite[thm 3.6.1 p.~119]{Bog} ~\cite[Prop 3.68 p.~28]{HairerSPDE}. This yields a first contradiction and implies that $F(\Phi+h)=F(\Phi) $ for all distributions $\Phi\in H^{\frac{2-d}{2}-\varepsilon}(M)$ and all $h\in C^\infty_c(A^c)$ which is much stronger than the almost sure statement. It means that $G\circ\Pi_\Lambda$ should not depend on $h\in C_c^\infty(A^c)$. More precisely, by definition of the support of a functional (see Definition~\ref{def:suppfonc}), for any such $h$ we have $G(\Phi_\Lambda+\Pi_\Lambda h)=G(\Phi_\Lambda)$, which implies that $\D_\Phi(G\circ\Pi_\Lambda)(h)=0$. Observing that 
\begin{align*}
    \D_\Phi(G\circ\Pi_\Lambda)(h)=\lim_{t\searrow0}\frac{G(\Phi_\Lambda+t\Pi_\Lambda h)-G(\Phi_\Lambda)}{t}=\D_{\Phi_\Lambda}(G)(\Pi_\Lambda h)\,,
\end{align*}
we conclude that $\D_{\Phi_\Lambda}(G)(\Pi_\Lambda h)=0$. In view of Corollary~\ref{Lemm21} taking $B=A^c$ and $T=\D_{\Phi_\Lambda}(G)$, we thus have that $\Pi_\Lambda\D_{\Phi_\Lambda}(G)=0$. In particular, for all $f\in C^\infty(M)$, 
\begin{align*}
   \Pi_\Lambda\D_{\Phi_\Lambda}(G)(f)=\D_{\Phi_\Lambda}(G)(\Pi_\Lambda f)=\D_\Phi(G\circ\Pi_\Lambda)(f)=0,
\end{align*}
so $G\circ\Pi_\Lambda $ is indeed constant.

\section{The cut-off GFF on cylinders is not reflection positive}\label{sec2}

In this section, we give an explicit counter-example contradicting the fact that $\Phi_\Lambda$ can be reflection positive on $\mathbb R^d$ and more generally on any Riemannian cylinder $M$ of the form $M=\mathbb{R}\times \Sigma$ where $\Sigma$ is complete Riemannian and the cylinder is endowed with the split metric $dt^2+g_\Sigma$, based on the following
\begin{lemm}[\cite{GJ}, Theorem 6.2.2]\label{lemm31}
A Gaussian random field $\Psi$ with covariance $C$ is reflection positive if and only if its covariance is reflection positive, in the sens that for any $f\in C_c^\infty(\mathbb{R}_{\geqslant 0}\times\Sigma)$,
\begin{align*}
    \langle \Theta f,Cf\rangle_{L^2(M)}\geqslant0.
\end{align*}
\end{lemm}
We construct hereafter such a function $f$ such that $\big\langle \Theta f,\frac{\Pi_\Lambda(\Delta_g)}{\Delta_g+1}f\big\rangle_{L^2(M)}<0$, which implies
\begin{thm}\label{thm2}
    The spectrally cut-off massive Gaussian free field on $\mathbb R^d$ or on any Riemannian cylinder $\Phi_\Lambda$ is not reflection positive.
\end{thm}
This implies that the spectrally cut-off GFF cannot be 
Markov. To construct our counter-example, we need the following Lemma.
\begin{lemm}\label{lemm2}
Let $\kappa\in[1, 1+1]$. There exists a function~$h\in C_c^{\infty}(\mathbb{R}_{\geqslant 0})$ such that
  \begin{equation*}
    \int_{-1}^{1}\ol{\wh{h}(-\xi)} \,\wh{h}(\xi) \, \frac{d\xi}{\xi^2+\kappa} < 0.
    \label{}
  \end{equation*}
\end{lemm}
\begin{proof}
  Note that because $h$ is real, one has~$\ol{\wh{h}(-\xi)}=\wh{h}(\xi)$. Thus~$h$ just needs to satisfy
  \begin{equation*}
    \int_{-1}^{1}\wh{h}(\xi)^2 \frac{d\xi}{\xi^2+\kappa}<0.
    \label{}
  \end{equation*}
  Now write
  \begin{equation*}
    \wh{h}(\xi)=A_h(\xi)+\ii B_h(\xi)\,,
    \label{}
  \end{equation*}
  where
  \begin{equation*}
    A_h(\xi)\defeq\int_{}^{} h(x)\cos(\xi x)dx\,,\quad B_h(\xi)\defeq\int_{}^{}h(x)\sin(\xi x)dx\,,
    \label{}
  \end{equation*}
  so that~$A_h$ is even and~$B_h$ is odd. This gives
  \begin{equation*}
    \wh{h}(\xi)^2=A_h(\xi)^2-B_h(\xi)^2+2\ii \underbrace{A_h(\xi)B_h(\xi)}_{\textrm{integral }=~0,\textrm{ odd}}.
    \label{}
  \end{equation*}
  The important idea is that since~$\wh{h}(0)=\int_{}^{}h(x)dx>0$, one needs to shift the weight of integration~$(\xi^2+\kappa)^{-1}$ ``away from zero''. This can be done by taking derivatives: setting $h=\varphi^{(2n)}$, we have
  \begin{align*}
    \int_{-1}^{1}\wh{\varphi^{(2n)}}(\xi)^2\frac{d\xi}{\xi^2+\kappa}&=2\int_{0}^{1}\frac{(\ii \xi)^{4n}}{\xi^2+\kappa} \big(A_{\varphi}(\xi)^2-B_{\varphi}(\xi)^2\big) d\xi   \\
    &=2\int_{0}^{1}\frac{\xi^{4n}}{\xi^2+\kappa} \big(A_{\varphi}(\xi)^2-B_{\varphi}(\xi)^2\big) d\xi.
  \end{align*}
Note that we have taken care of choosing a number of derivatives that makes~$(\ii \xi)^{4n}$ real. Now, observe that
  \begin{equation*} \frac{ \int_{0}^{1}\frac{\xi^{4n}}{\xi^2+\kappa} \big(A_{\varphi}(\xi)^2-B_{\varphi}(\xi)^2\big) d\xi}{\int_{0}^{1}\frac{\xi^{4n}}{\xi^2+\kappa} d\xi} \lto \big(A_{\varphi}(1)^2-B_{\varphi}(1)^2\big)\,\quad \textrm{as }n\lto \infty.
  \end{equation*}
Choosing $\varphi$ supported near $\frac{\pi}{2}$ ensures that $\cos(x)\approx 0$ and $\sin(x)\approx 1$, so that $A_{\varphi}(1)^2-B_{\varphi}(1)^2<0$. Taking $n$ large enough but finite and rescaling $\varphi$, one finally obtains the desired result. Moreover, since $\varphi$ is compactly supported on $\mathbb{R}_{\geqslant 0}$, so is $h$. 
\end{proof}
\begin{proof}[of Theorem \ref{thm2}]
As alluded to, the proof follows from the fact that we provide a counter-example of the positivity of the covariance, that is to say, we construct $f\in  C_c^\infty(\mathbb{R}_{\geqslant 0}\times\Sigma)$ such that  $ \big\langle \Theta f,\frac{\Pi_\Lambda(\Delta_g)}{\Delta_g+1}f\big\rangle_{L^2(M)}<0$. 

We denote by $E(\lambda)d\lambda$ the projection valued measure of the slice Laplacian $\Delta_\Sigma$ which is well--known to be self-adjoint by the completeness of $\Sigma$. To any function $f\in L^2(M)$ 
on the cylinder $M=\mathbb{R}\times \Sigma$, we will denote by 
$$
\widehat{f}(\tau,\lambda)= \int e^{-it\tau} E(\lambda)\left( f(t,.) \right) dt  
$$ 
its Fourier transform w.r.t. the time variable $t$ and its spectral transform with respect to the spectral measure of $\Delta_\Sigma$. Then we rewrite the pairing as
\begin{eqnarray*}
\left\langle \Theta f,\frac{\Pi_\Lambda(\Delta_g)}{\Delta_g+1}f\right\rangle_{L^2(M)}
=\int_{\{\tau^2\vee\lambda^2\leqslant\Lambda\}\subset\mathbb{R}\times \mathbb{R}_{\geqslant 0}}\frac{1}{\tau^2+\lambda^2+1} \left(\int_\Sigma \overline{\widehat{f}(-\tau,\lambda)}\widehat{f}(\tau,\lambda)dV \right)  d\tau d\lambda \\
=\int_{-\Lambda}^{\Lambda}\left( \int_{0}^{\Lambda}\frac{1}{\tau^2+\lambda^2+1} \left(\int_\Sigma \overline{\widehat{f}(-\tau,\lambda)} \, \widehat{f}(\tau,\lambda) \, dV \right)  d\lambda \right) d\tau 
\end{eqnarray*}

Now we will reduce to the one variable case by a scaling argument. Let $\varphi\in C^\infty_c(\mathbb{R}_{\geqslant 0})$ such that
$ \int_{-1}^1 \frac{\widehat{\varphi}^2(\tau)}{\tau^2+\lambda^2 + \frac{1}{\Lambda^2}}d\tau<0 $ for all $\lambda\in [0,1]$. The existence
of such $\varphi$ comes from Lemma~\ref{lemm2}. 
Choose $\chi\in L^2(\Sigma)$ such that $\int_\Sigma\vert\chi\vert^2=1$ and $\chi$ has non trivial spectral measure in the interval $[0,1]$:
$$
\int_0^1 \big\Vert E(\lambda)(\chi)\big\Vert^2_{L^2(\Sigma)} d\lambda > 0.
$$ 
Then set
$$
f=\int_0^\infty \Lambda \varphi(\Lambda t)E\left(\frac{\lambda}{\Lambda}\right)(\chi) d\lambda.
$$
Replacing the expression of $f$ in the pairing we get
\begin{eqnarray*}
\left\langle \Theta f,\frac{\Pi_\Lambda(\Delta_g)}{\Delta_g+1}f\right\rangle_{L^2(M)} &=& \int_{0}^{\Lambda}  \Big\Vert E\left(\frac{\lambda}{\Lambda}\right)( \chi)\Big\Vert_{L^2(\Sigma)}^2   \left(\int_{-\Lambda}^{\Lambda} \frac{\widehat{\varphi}(\frac{\tau}{\Lambda})^2}{\tau^2+\lambda^2+1}   d\tau \right)  d\lambda \\
&=&\Lambda^2 \int_{0}^{1}  \big\Vert E\left(\lambda\right)( \chi)\big\Vert_{L^2(\Sigma)}^2   \left(\int_{-1}^1 \frac{\widehat{\varphi}(\tau)^2}{\Lambda^2(\tau^2+\lambda^2)+1}   d\tau \right)  d\lambda \\
&=&\int_{0}^{1}  \big\Vert E\left(\lambda\right)( \chi)\big\Vert_{L^2(\Sigma)}^2   \left(\int_{-1}^1 \frac{\widehat{\varphi}(\tau)^2}{\tau^2+\lambda^2+\frac{1}{\Lambda^2}}   d\tau \right)  d\lambda <0 .
\end{eqnarray*}
We have thus proven that the cut-off covariance is not reflection positive, and by Lemma~\ref{lemm31}, $\Phi_\Lambda$ is therefore not reflection positive.
\end{proof}

\section{The cut-off GFF on compact manifolds is not reflection positive}\label{sec3}

As discussed in the introduction, it suffices to show that
the cut-off GFF is not reflection positive, this will also imply that it is not Markov.
\begin{thm}\label{thm41}
Assume that $(M,g)$ is a smooth compact Riemannian manifold which is reflection positive. 
Then there exists $L>0$ such that for any $\Lambda\geqslant L$, there exists a function $f\in C^\infty_c(M_+)$
such that 
\begin{eqnarray*}
\left\langle \Theta f, \frac{\Pi_\Lambda(\Delta_g)}{\Delta_g+1}f \right\rangle_{L^2(M)}<0.
\end{eqnarray*}
\end{thm} 
\begin{proof}
First, by commutativity of $\Delta_g$ and $\Theta$, note that we can choose the $L^2$ basis of $\Delta_g$ in such a way that any eigenfunction $e_\lambda$ of the Laplacian $\Delta_g$ satisfies either 
$\Theta e_\lambda=e_\lambda$ which we call even or $\Theta e_\lambda=-e_\lambda$ which we call odd. Observe that the odd eigenfunctions are non empty since the span of eigenfunctions is dense in $C^\infty(M)$ which contains odd functions hence odd eigenfunctions is a non empty subset of eigenfunctions. Hence we choose $L$ large enough so that there exists $\lambda\leqslant L^2$ such that $e_\lambda$ is odd. Given $\Lambda \geqslant L^2$, we note $\lambda_*$ the largest eigenvalue $\leqslant \Lambda^2$ with odd eigenfunction $e_{\lambda_*}$.

Second, consider the linear map
\begin{eqnarray*}
T: f\in C^\infty(M_+)\mapsto \left(\left\langle f,e_{\lambda} 
\right\rangle \right)_{\lambda\leqslant \Lambda^2}\in \mathbb{R}^{\dim(\EL)}.
\end{eqnarray*}
Then we would like to show that linear $T: C^\infty(M_+)\mapsto \mathbb{R}^{\dim(\EL)}$ is onto. Assume by contradiction it were not onto, then $T(C^\infty(M_+) )$ is a strict vector subspace of $\mathbb{R}^{\dim(\EL)}$ then we can choose some vector $(c_\lambda)_{\lambda\leqslant \Lambda^2}$ in $T(C^\infty(M_+) )^\perp$. In other words there exists a linear combination $\varphi=\sum_{\lambda\leqslant \Lambda^2} c_\lambda e_\lambda \in \EL, \varphi\neq 0$ such that for all $f\in C^\infty_c(M_+)$:
$$
\sum_{\lambda\leqslant \Lambda^2} c_\lambda \left\langle e_\lambda,f\right\rangle = 0, 
$$
meaning $\varphi\in\EL$ vanishes on $M_+$. But $M_+$ has nonempty interior which contradicts Lemma~\ref{thmnodal} hence the map $T$ is surjective. The surjectivity of $T$ ensures we can find $f\in C^\infty_c(M_+)$ such that $\left\langle f,e_\lambda\right\rangle = 0$ if $\lambda\neq \lambda_*$ and $\left\langle f,e_{\lambda_*}\right\rangle=1$. For this $f$, we calculate
\begin{eqnarray*}
\left\langle \Theta f, \frac{\Pi_\Lambda(\Delta_g)}{\Delta_g+1}f \right\rangle_{L^2(M)}=
\sum_{\lambda\leqslant \Lambda^2} \left\langle f,\Theta e_\lambda\right\rangle \frac{1}{\lambda+1} \left\langle f,e_\lambda\right\rangle=
-\left\langle f,e_{\lambda_*}\right\rangle \frac{1}{\lambda_*+1} \left\langle f,e_{\lambda_*}\right\rangle=-\frac{1}{\lambda_*+1}<0\,,
\end{eqnarray*}
which concludes the proof.
\end{proof}

\section{A counterexample on $\mathbb{R}^{3}$}
\label{SectionCounterexample}

For $f\in L^2(\mathbb{R}^3)$ write
$$
(\rho\Pi_\Lambda\Phi)(f) \defeq \int_{\mathbb{R}^3} f(x) \rho(x) (\Pi_\Lambda\Phi)(x) \,g dx.
$$ 

We would like to summarize in one key lemma the central idea behind our counterexample.
\begin{lemm}\label{keylemma}
Let $A\subset \mathbb{R}^d$ be a closed subset with nonempty interior and 
$B\subset \mathbb{R}^{d*}$ a closed compact ball in Fourier space. Denote by $L^2_A(\mathbb{R}^d)$ the subspace of $L^2$ functions whose supported is contained in $A$.
Then the Fourier restriction map defined as 
$$T:\varphi\in L^2_A(\mathbb{R}^d) \longmapsto \widehat{\varphi}|_{B}\in L^2(B)$$ 
has everywhere dense image.
\end{lemm}
The idea of proof is very similar to the one of Thm~\ref{thm41}. The simple but powerful idea is that any function on the ball $B$ can be approximated by the Fourier transform of some function supported over $A$.
\begin{proof}
By contradiction, if $\text{Ran}(T)$ were not dense, then $\text{Ran}(T)^\perp$ would be a non-empty closed vector subspace and there would be some non-null element $g\in L^2(B)$ such that for all $f\in L^2_A$
\begin{eqnarray*}
0 = \big\langle \widehat{f},g\big\rangle _{L^2(B)} = \big\langle \widehat{f},\chi_{B}g\big\rangle_{L^2(\mathbb{R}^d)} = \left\langle  f, \mathcal{F}^{-1}\left( \chi_{B}g\right)\right\rangle_{L^2(\mathbb{R}^d)},
\end{eqnarray*}
from Plancherel identity. The function $\mathcal{F}^{-1}\left( \chi_{B}g\right)$ is therefore not supported on $A$ hence vanishes on some open subset. This contradicts the fact that $\mathcal{F}^{-1}\left( \chi_{B}g\right)$ is a non-trivial analytic function by the Paley-Wiener Theorem.
\end{proof}

We emphasize here the dependence of the measures $\nu_{\rho,\Lambda}$ on the coupling constant $c$ by writing $\nu_{c,\rho,\Lambda}$. Write as well $G_{c,\rho,\Lambda}$ for $G_{M,\rho,\Lambda}$ in \eqref{EqGWholeSpace}. Denote by $(x_1,x_2,x_3)$ the canonical coordinates of a point $x\in\mathbb{R}^3$. Last write $f\in L^2(\mathbb{R}^3_+)$ to mean that $f\in L^2(\mathbb{R}^3)$ and $\text{supp}(f)\subset\{x_1>0\}$.

\begin{thm} \label{ThmNotRFRegularizedPhi43}
For $c>0$ small enough there exists a function $f\in L^2(\mathbb{R}^{3}_+)$ such that
\begin{eqnarray}
\mathbb{E}_{\nu_{c,\rho,\Lambda}}\left[  \overline{(\rho\Pi_\Lambda\Phi)(\Theta f)} \, (\rho\Pi_\Lambda\Phi)(f) \right] < 0.
\end{eqnarray}
This implies that the regularized $\Phi^4_3$ measures $\nu_{c,\rho,\Lambda}$ are not reflection positive for $c>0$ small enough, depending on $\Lambda$ and $\rho$.
\end{thm}

\begin{proof}
Since $0\leq G_{c,\rho,\Lambda}(\varphi)$ is converging to $1$ as $c>0$ goes to $0$ for all distributions $\varphi$ one has from dominated convergence
\begin{equation} \label{EqConvergenceSmallCoupling}
\mathbb{E}_{\nu_{c,\rho,\Lambda}}\left[  \overline{(\rho\Pi_\Lambda\Phi)(\Theta f)} \, (\rho\Pi_\Lambda\Phi)(f) \right] \underset{c\rightarrow 0^+}{\longrightarrow} \mathbb{E}_{\mu_{\text{GFF}}}\left[ \overline{(\rho\Pi_\Lambda\Phi)(\Theta f)} \, (\rho\Pi_\Lambda\Phi)(f) \right].
\end{equation}
To justify the use of the dominated convergence at some fixed cut-off $\rho$ and $\Lambda$ we note the lower bound on the interaction
$$ 
c\Vert\rho\Phi_\Lambda \Vert^4_{L^4(\mathbb{R}^3)}+ca_\Lambda \Vert\rho\Phi_\Lambda \Vert^2_{L^2(\mathbb{R}^3)} \geqslant c\int_{\mathbb{R}^3} \rho^4\Phi_\Lambda^4 - c\frac{\vert a_\Lambda\vert}{2\delta^2} \text{Vol}(\text{supp}(\rho))-c\frac{\vert a_\Lambda \vert\delta^2}{2}\int_{\mathbb{R}^3} \Phi_\Lambda^4\rho^4,  
$$
for any $\delta>0$, using Young's inequality and the compactness of the support of $\rho$. Choosing $\delta$ small enough yields a lower bound of the form
$$
c\Vert\rho\Phi_\Lambda \Vert^4_{L^4(\mathbb{R}^3)}+ca_\Lambda \Vert\rho\Phi_\Lambda \Vert^2_{L^2(\mathbb{R}^3)}  \geqslant c \frac{2}{3} \Vert\rho\Phi_\Lambda \Vert^4_{L^4(\mathbb{R}^3)}-cK 
$$
for some $K>0$.

From the convergence result in \eqref{EqConvergenceSmallCoupling} it suffices to find $f\in L^2(\mathbb{R}^{3}_+)$ such that
$$ 
\mathbb{E}_{\mu_{\text{GFF}}}\left[ \overline{(\rho\Pi_\Lambda\Phi)(\Theta f)} \, (\rho\Pi_\Lambda\Phi)(f) \right] =  \left\langle {\color{black} \psi_\Lambda(\sqrt{\Delta})} (\rho \overline{\Theta f}) \, , \, (\Delta+1)^{-1}\big(\psi_\Lambda(\sqrt{\Delta}) (\rho f)\big) \right\rangle_{L^2} < 0
$$ 
and choose $c>0$ small enough. Set the linear map 
\begin{eqnarray*}
T: f\in L^2(\mathbb{R}^{3}_+)\longmapsto \widehat{\rho f} \chi_{B(0,\Lambda)}\in L^2\big(B(0,\Lambda)\big)
\end{eqnarray*}
where on the right hand side we consider the restriction of the Fourier transform $\widehat{\rho f}$ to the Fourier ball of radius $\Lambda$. By a similar argument as in Lemma~\ref{keylemma}, the key idea is to prove that the image of $T$ is dense in $L^2(B(0,\Lambda))$. If it were not true $\text{Ran}(T)^\perp$ would be a non-empty closed vector subspace and there would be some non-null element $g\in L^2(B(0,\Lambda))$ such that for all $f\in L^2(\mathbb{R}^{3}_+)$
\begin{eqnarray*}
0 = \big\langle \widehat{\rho f},g\big\rangle _{L^2(B(0,\Lambda))} = \big\langle \widehat{\rho f},\chi_{B(0,\Lambda)}g\big\rangle_{L^2(\mathbb{R}^3)} = \left\langle  f, \rho\mathcal{F}^{-1}\left( \chi_{B(0,\Lambda)}g\right)\right\rangle_{L^2(\mathbb{R}^3)},
\end{eqnarray*}
from Plancherel identity. The function $\rho\mathcal{F}^{-1}\left( \chi_{B(0,\Lambda)}g\right)$ is therefore not supported on the half-space $\{x_1>0\}$.
Since $\rho>0$ on some non-empty region of $\{x_1>0\}$ this implies that the function $\mathcal{F}^{-1}\left( \chi_{B(0,\Lambda)}g\right)$ vanishes in some open subset contained in the half-space $\{x_1>0\}$. This contradicts the fact that $\mathcal{F}^{-1}\left( \chi_{B(0,\Lambda)}g\right)$ is a non-trivial analytic function by the Paley-Wiener Theorem.

So the provisional conclusion is that there exists a sequence $f_n$ in $ L^2(\mathbb{R}^{3}_+)$ such that the sequence $T(f_n)$ converges in $L^2(B(0,\Lambda))$ to $\xi_1 \chi_{B(0,\Lambda)}(\xi)$. Therefore for this precise sequence $f_n$ we find that
\begin{equation*} \begin{split}
\left\langle \psi_\Lambda(\sqrt{\Delta}) (\rho \overline{\Theta f}) \, , \, (\Delta+1)^{-1}\right.&\left.\big(\psi_\Lambda(\sqrt{\Delta}) (\rho f)\big) \right\rangle_{L^2}   \\
&= \int_{\mathbb{R}^{3*}} \psi_\Lambda(\vert \xi\vert)^2 \big(\vert \xi\vert^2+1\big)^{-1} \, \overline{T(f_n)}(-\xi_1,\xi_2,\xi_3) \, T(f_n)(\xi_1,\xi_2,\xi_3) \, d\xi   \\
\end{split} \end{equation*} 
is converging as $n$ goes to $\infty$ to
$$
\int_{\mathbb{R}^{3*}} \psi_\Lambda(\xi)^2 \big(\vert \xi\vert^2+1\big)^{-1}(-\xi_1^2)\chi_{B(0,\Lambda)}^2\,d\xi = -\int_{\mathbb{R}^{3*}} \psi_\Lambda(\xi)^2 \big(\vert \xi\vert^2+1\big)^{-1}\xi_1^2\,d\xi < 0.
$$
This shows indeed that the heat regularized (and space localized) Gaussian free field measure is not reflection positive. Of course the Gaussian free field measure itself is reflection positive and the above proof breaks down when $\Lambda=+\infty$ as the function $\xi_1$ is no longer an elelment of $L^2$.
\end{proof}

Theorem \ref{ThmNotRFRegularizedPhi43} \textbf{does not exclude} the fact that $\nu_{c,\rho,\Lambda}$ may be reflection positive for some coupling constant $c$ that would be large enough.

\end{document}